\documentclass[amsthm]{elsart}
\usepackage{amsmath, amsthm, latexsym, amssymb, amsfonts}
\usepackage[pdftex]{graphicx,color} 
\usepackage[all]{xy}

 \theoremstyle{plain}
 \newtheorem{tm}{Theorem}[section]
 \newtheorem{lm}[tm]{Lemma}
\newtheorem{coro}[tm]{Corollary}
 \newtheorem{pro}[tm]{Proposition}

 \theoremstyle{definition}
 \newtheorem{defi}[tm]{Definition}
 \newtheorem{rema}[tm]{Remark}
 
  \newtheorem{ex}[tm]{Example}

 \newtheorem*{th*}{Theorem}

\newcommand{\cl}[1]{\mathcal{#1}}

\newcommand{\Z}{\mathbb Z}
\newcommand{\N}{\mathbb N}
\newcommand{\R}{\mathbb R}
\newcommand{\F}{\mathbb F}
\newcommand{\Pp}{\mathbb P}

\newcommand{\D}{\Delta}
\newcommand{\cA}{\cl A}

\newcommand{\cC}{\cl C}

\newcommand{\cL}{\cl L}

\newcommand{\cO}{\cl O}

\newcommand{\p}{\partial}
\newcommand{\T}{{\mathbb T}}

\newcommand{\K}{{\mathbb K}}

\newcommand{\ov}{\overline}
\newcommand{\la}{\langle}
\newcommand{\ra}{\rangle}

\newcommand{\sig}{\sigma}
\newcommand{\Sig}{\Sigma}

\newcommand{\m}{\mathfra\K{m}}

\newcommand{\Supp}{\operatorname{Supp}}

\newcommand{\codim}{\operatorname{codim}}

\newcommand{\Pic}{\operatorname{Pic}}

\newcommand{\Hom}{\operatorname{Hom}}

\newcommand{\depth}{\operatorname{depth}}
\newcommand{\reg}{\operatorname{reg}}
\newcommand{\Cl}{\operatorname{Cl}}
 \newcommand{\Vol}{\operatorname{Vol}}
  \newcommand{\Ehr}{\operatorname{Ehr}}

    \newcommand{\Gal}{\operatorname{Gal}}

\newcommand{\dis}{\displaystyle}
\def\aa{{\bf \alpha}}
\def\bb{\beta}
\def\a{{\bf a}}
\def\t{{\bf t}}
\def\uu{{\bf u}}
\def\vv{{\bf v}}
\def\x{{\bf x}}
\def\m{{\bf m}}
\def\ev{{\text{ev}}}

\newcommand{\cK}{ \K^r}

\newcommand{\rs}[1]{Section~\ref{S:#1}}

\newcommand{\rp}[1]{Proposition~\ref{P:#1}}

\newcommand{\re}[1]{(\ref{e:#1})}

\newcommand{\rt}[1] {Theorem~\ref{T:#1}}
\newcommand{\rd}[1]{Definition~\ref{D:#1}}

 \usepackage{graphics}
 \usepackage[english]{babel}
\usepackage[latin1]{inputenc}
\usepackage{times}
\usepackage[T1]{fontenc}
 \usepackage{graphics}
\usepackage{pst-all}

\begin{document}
\begin{frontmatter}

%%%%%%%%%%%%%%%%%%%%%%%Topmatter

\title{Multigraded Hilbert function and toric complete intersection codes}
 \thanks{The first author is supported by T\"{U}B\.{I}TAK-2219, the second author is partially supported by NSA Grant H98230-13-1-0279}
\author[Mesut Sahin]{Mesut \c{S}ahin}
\address[Mesut Sahin]{Department of Mathematics, Hacettepe University, Ankara, TURKEY}
\ead{mesut.sahin@hacettepe.edu.tr}
\author[Ivan Soprunov]{Ivan Soprunov}
\address[Ivan Soprunov]{Department of Mathematics, Cleveland State University, Cleveland, OH USA}
\ead{i.soprunov@csuohio.edu}
%\keywords{Evaluation code, lattice polytope, Ehrhart polynomial, sparse polynomial system}
%\subjclass[2010]{Primary 14M25, 14G50; Secondary 52B20}

%\date{}

\begin{abstract} Let $X$ be a complete 
$n$-dimensional simplicial toric variety with homogeneous coordinate ring $S$.
We study the multigraded Hilbert function $H_Y$ 
of reduced $0$-dimensional subschemes $Y$ in $X$.
We provide explicit formulas and prove non-decreasing and stabilization
properties of $H_Y$ 
when $Y$  is a $0$-dimensional complete intersection in $X$.
We  apply our results to computing the dimension of some evaluation codes on $0$-dimensional complete intersection in simplicial toric varieties.

\end{abstract}

\begin{keyword}
 evaluation codes, toric varieties, multigraded Hilbert function
\end{keyword}

\end{frontmatter}

%%%%%%%%%%%%%%%%%%%%%%%%%%%%%%%%%%%%%%%%%%%%%%%%%%%%%%%%%%
\section{Introduction}  

Let $X$ be a complete simplicial toric variety of dimension $n$ over an algebraically closed field $\K$, 
and $\dis S=\K[x_1,\dots, x_r]=\bigoplus_{\aa \in \cA} S_{\aa}$ 
be its homogeneous coordinate ring, multigraded by $\cA \cong \Cl(X)$, and let $B\subset S$ be its irrelevant ideal. Given a reduced closed subscheme 
$Y\subset X$ there is a unique radical $B$-saturated ideal $I(Y)\subset S$
which defines $Y$ (see \cite{Coxhom}). In this paper we study the Hilbert function 
$H_Y(\aa)=\dim_{\K}(S/I(Y))$ where $Y$ is a reduced $0$-dimensional subscheme of $X$.
Although the general situation was previously considered in fundamental work by Maclagan and Smith 
 \cite{MacS4, MacS5}, our goal is to better understand
 the multigraded regularity of $Y$ (the set of degrees $\aa$ where $H_Y(\aa)=|Y|$) and
 to provide an explicit combinatorial formula for $H_Y$ in terms of the polytopes $P_{\aa_i}$
 of the generators of $I(Y)$.
 
In particular, when $Y$ is a $0$-dimensional complete intersection lying in the dense orbit  $\T^n$ of $X$, 
our results imply the following (see \rp{HilbFla} and Corollary \ref{HF}).
\begin{th*} Let $Y\subseteq \T^n$ be a reduced complete intersection subscheme of $X$ such that $I(Y)$ is generated by $n$ homogeneous polynomials with semi-ample degrees $\aa_1,\dots,\aa_n$. 
Then for any $\aa\in\cA$ we have
$$
H_Y(\aa)=\sum_{s=0}^{n}\,\, (-1)^s\!\!\! \sum_{I\subseteq \{1,\dots,n\}, |I|=s} |P_{\aa-\aa_I} \cap M |,
$$
where $\aa_I=\sum_{i\in I}\aa_i$. Moreover,\\
(\textbf{i}) if $\aa\not\succeq \aa_i$ then $H_Y(\aa)=|P_{\aa}\cap M|$,\\
(\textbf{ii}) $H_Y(\aa)\leq H_Y(\aa'), \text{ for all }\aa \preceq \aa'$,\\ %$\aa \in\N\beta$,\\
(\textbf{iii}) $H_Y(\aa)=|Y|$ for all  $\aa \succeq \aa_1+\cdots+\aa_n$.
\end{th*}

In the above theorem, we write $\aa \preceq \aa'$ if and only if $\aa'-\aa$ lies in $\N\beta$, where $\N\beta\subset \cA$ denotes the semigroup generated by the degrees of the variables $x_i$. 

As an application, we compute the dimension for a class of evaluation codes on $0$-dimensional
complete intersections in a toric variety (see \rt{CodeFla}), answering a question posed in \cite{sop}. The bound given in (\textbf{iii}) above is used to eliminate trivial codes. Our results generalize
the work of \cite{DRT} who used the classical Hilbert function to compute the dimension of 
evaluation codes on $0$-dimensional complete intersections in the projective space.

The paper is organized as follows. In \rs{prel} we collect all necessary preliminaries as well as give 
certain sufficient conditions for a $0$-dimensional subscheme of $X$ to be a complete intersection. 
%In \rs{lci} we give a characterization of a reduced $0$-dimensional subscheme lying in the torus to be a local complete intersection in terms of solution
%sets of Laurent polynomial systems with fixed Newton polytopes, thus putting 
%toric complete intersection as defined in \cite{sop} in the framework of ideals in the homogeneous coordinate ring.
\rs{mhf} contains the main results on the multigraded Hilbert function of $0$-dimensional subschemes of $X$. It builds on and generalizes the existing literature and improves the known bound on the multigraded regularity. %utilizing ideas from algebraic geometry, commutative algebra and combinatorics.

This allows application of our results on the multigraded Hilbert function to the dimension calculation
for toric complete intersection codes, which we discuss in \rs{codes}.

%By taking functions from a finite-dimensional subspace $\cL$ of the Laurent polynomial ring $\F_q[t_1^{\pm 1},\dots,t_n^{\pm 1}]$ over the field $\F_q$ and evaluating them at points in a finite subset  $Y=\{p_1,\dots,p_N\}$ of $(\F_q^*)^n$, we define the {\it evaluation map}
%$$\text{ev}_{Y}:\cL\to \F_q^N,\quad f\mapsto (f(p_1),\dots,f(p_N)).$$
%The image of $\text{ev}_{Y}$ is a linear code, called the {\it evaluation code}, and denoted by $\cC_{Y,\cL}$. 
%
%In this paper we study evaluation codes $\cC_{Y,\cL}$ where  $\cL$ is a
%space of Laurent polynomials and $Y$ is a zero-dimensional complete intersection of $n$ hypersurfaces
%in a toric variety $X$ of dimension $n$.
%
%\cite{DRT} discuss how one can use the classical Hilbert function to compute the dimension of the corresponding evaluation codes $\cC_{Y,\cL_d}$, when $X=\Pp^n(\F_q)$, the set $Y$
%is an arbitrary zero-dimensional complete intersection in $X$, and $\cL=\cL_d$ is the space of homogeneous
%polynomials of degree $d$. 

\section{Preliminaries} \label{S:prel} In this section, we recall some standard definitions and results from toric geometry and fix some notation. For all unexplained concepts and for more details we refer the reader to the wonderful books \cite{F} by Fulton or \cite{CLSch} by Cox, Little and Schenck. Although these books study toric varieties over the complex numbers, the results we will be using carry easily to toric varieties over algebraically closed fields. Throughout the paper $\K$ is an arbitrary algebraically closed field and $\T=\K^*$
is its multiplicative group. Given a vector $\uu\in\Z^r$ we use $x^\uu$ to denote the Laurent monomial 
$x^\uu=x_1^{u_1}\dots x_r^{u_r}$. 
We also use $[m]$ to denote the set $\{1,\dots,m\}$ for any positive integer $m\geq 1$.

Let $M$ be a lattice of rank $n$ and $N$ be the dual lattice. By $M_{\R}$ (resp. $N_{\R}$) we denote the corresponding
real $n$-dimensional vector space $M\otimes\R$ (resp. $N\otimes{\R}$).
Let $\Sig \subseteq N_\R$ be a complete rational polyhedral fan and $X$ be the $n$-dimensional complete toric variety it determines. We assume that 
$\Sig$ (and, hence, $X$) is simplicial without loss of generality, where needed we may 
also assume that $X$ is smooth.

Recall the definition of the homogeneous coordinate ring $S$ of the toric variety $X$.
Denote by $\rho_1,\dots,\rho_r$ the rays in $\Sig$ and $\vv_1,\dots,\vv_r\in N$ the corresponding primitive lattice vectors generating them. Each $\rho_i$ gives rise to a prime torus-invariant Weil divisor $D_i$.
We introduce one variable $x_i$ for each $D_i$ and consider the polynomial ring $S=\K[x_1,\dots,x_r]$. 
Also, for each cone $\sig \in \Sig$ we denote $\dis x^{\hat{\sig}}=\Pi_{\rho_i \notin \sig}^{}x_i$. This gives rise to the irrelevant ideal $B=\la x^{\hat{\sig}} ~:~ \sig \in \Sig\ra$ in $S$.

The collection $\vv_1,\dots,\vv_r\in N$ defines a map $\phi:M\to\Z^r$ which sends $\m\in M$ to the vector $(\la\m,\vv_1\ra,\dots, \la\m,\vv_r\ra)\in\Z^r$. The class group $\Cl(X)$ of Weil divisors on $X$ modulo linear equivalence is generated by the classes of $D_1,  \dots, D_r$ (see \cite[Sec.3.4]{F}). Since $X$ is complete $\vv_1,\dots,\vv_r$ spans $N_\R$ which leads to the 
following exact sequence:
$$\dis \xymatrix{ \mathfrak{P}: 0  \ar[r] & M \ar[r]^{\phi} & \Z^r \ar[r]^{\deg} & \Cl(X) \ar[r]& 0},$$
where the degree map sends $(a_1,\dots,a_r)$ to the class $[\sum_j a_jD_j]$.\\
 Applying  $Hom(-,\T)$ we get the dual exact sequence
 $$\dis \xymatrix{ \mathfrak{P}^*: 1  \ar[r] & \Hom(\Cl(X),\T) \ar[r]^{~~~~~~~~~~i} & \T^r \ar[r]^{\pi~~~~~~~~} & \Hom(M,\T) \ar[r]& 1}.$$It is convenient to write the above sequences in a more explicit way by fixing a coordinate system. For this
choose a basis $\{e_1,\dots, e_n\}$ for $M$  and the dual basis $\{e^*_1,\dots, e^*_n\}$ for $N$.
This allows us to represent elements of $M$ by column vectors and elements of $N$ by row vectors in $\Z^n$.
The map $\phi$ is now given by an integer matrix with rows  $\vv_1,\dots, \vv_r$. 
We thus obtain
$$\dis \xymatrix{ \mathfrak{P}: 0  \ar[r] & \Z^n \ar[r]^{\phi} & \Z^r \ar[r]^{{\deg}} & \cA \ar[r]& 0},$$  
where $\cA=\Z^r/\phi(\Z^n)\cong\Cl(X)$ and $\deg$ is the projection map.

%Similarly, let $\cM_{\deg}$ be the matrix of the degree map. 
% WHEN THERE IS TORSION DEGREE MAP IS NOT JUST A MATRIX!

As for the dual sequence, we get an isomorphism $\Hom(M,\T)\cong \T^n$ and  the map $\pi$
becomes
$$\pi:(\xi_1,\dots, \xi_r)\mapsto(\xi^{\uu_1},\dots, \xi^{\uu_n}),$$
where $\uu_1,\dots, \uu_n$ are the columns of $\phi$.  We get 
$$\dis \xymatrix{ \mathfrak{P}^*: 1  \ar[r] & G \ar[r]^{i} & \T^r \ar[r]^{\pi} & \T^n \ar[r]& 1}.$$
The map $i$ is an injection and we identify $G=\Hom(\cA,\T)$ with the kernel of $\pi$.

Recall the definition of $X$ as a GIT quotient. The map $\pi$ from $\mathfrak{P}^*$ 
is in fact the restriction of the geometric quotient map (which we also denote by $\pi$)
$$\pi: \cK\setminus V(B) \rightarrow X$$ 
to the torus $\T^r\subset\cK$. Here $V(B)$ denotes the affine subvariety defined by the irrelevant ideal $B$, so  $V(B)$ is a union of coordinate subspaces. 

The exact sequence $\mathfrak{P}$ endows the ring $S$ with a 
multigrading  by setting degrees of variables as $\bb_j:=\deg_{\cA}(x_j):=\deg(e_j)$, where $e_j$ is the standart basis element of $\Z^r$ for each $j\in[r]$. 
Thus, $\dis S=\bigoplus_{\aa \in \cA} S_{\aa}$, where $S_{\aa}$ denotes the set of homogeneous polynomials in $S$ of multidegree $\aa$. We note that the degrees of homogeneous polynomials of $S$ actually lie in the semigroup $\N\bb$ generated by $\bb_1,\dots,\bb_r$, i.e. $\dim_{\K}{S}_{\aa}=0$ when $\aa \notin \N\bb$. 

The isomorphism $\cA \cong \Cl(X)$ allows us to speak of degrees $\alpha\in\cA$ lying in 
$\Pic(X)\subseteq\Cl(X)$.
%
%By abuse of notation we let $\pi$ denote the geometric quotient $\pi: \cK\setminus V(B) \rightarrow X$.   Denoting by $:=\deg_{\cA}(x_i)=\deg(e_i)$ the degrees of the variables, %
%We remark that in the language of toric geometry $M$ is the character lattice of the algebraic torus $T$ and
 %$N$ is the lattice of one-parameter subgroups. 
%The Picard group $\Pic(X)$ of Cartier divisors on $X$ modulo linear equivalence is a subgroup of 
%$\Cl(X)$, hence, we may
%speak of degrees $\alpha\in\cA$ lying in $\Pic(X)$ using the isomorphism $\cA \cong \Cl(X)$.  
%
Furthermore, given a Cartier divisor $D$ on $X$ we say that 
it is {\it semi-ample} (or basepoint free as in \cite{CLSch}) if the corresponding line bundle $\cO(D)$ is generated by global sections.
Since the property of being ample (resp. semi-ample) is 
preserved under linear equivalence we may speak of ample (resp. semi-ample) degrees $\alpha$ in $\cA$.
The fan $\Sig$ of $X$ determines an important subsemigroup $\cl K$ of the semigroup $\N\bb$. Namely, $\dis \cl K=\cap_{\sig \in \Sig} \N\hat{\sig}$, where $\N\hat{\sig}$ is the semigroup generated by the subset $\{\bb_j \; : \; \rho_j \notin \sig\}$. 
Geometrically, $\N\bb$ corresponds to the subset of $\Cl(X)$ containing the classes of effective Weil divisors on $X$ and $\cl K$ corresponds to the subset of $\Pic(X)$ containing the classes of nef (numerically effective) line bundles on $X$. By \cite[Theorem 6.3.12]{CLSch}, $\cl K$ is the set of semi-ample degrees in $\N\bb \subseteq \cA$.

Now let  $D=\sum_{j=1}^r a_j D_j$ be a torus-invariant Weil divisor on $X$. %, where $D_j$ is the torus-invariant prime divisor given by $x_j=0$. 
It defines a rational polytope
 $$P_D:=\{\m\in M_\R\ :\  \la \m, \vv_j \ra \geq -a_j, ~~\forall j \in [r]\}.$$ 
Note that equivalent divisors in $\alpha\in\Cl(X)$ will have the same polytope up to a lattice translation. We thus write $P_{\aa}$ 
 to denote any $P_D$ for $D\in\alpha$. We will also write $P_{\aa}\subseteq P_{\aa'}$
 when $P_D\subseteq P_{D'}$ for some $D\in\aa$ and $D'\in\aa'$, and simply say that $P_{\aa}$ is contained
 in $P_{\aa'}$.  According to \cite[Sec. 1]{Coxhom}  $S_{\aa}$ is isomorphic to the space of global sections of the sheaf $\cO(D)$ which is spanned by characters of the torus $\T$:
  \begin{equation}\label{e:span}
S_\aa\cong \bigoplus_{\m\in P_\aa\cap M}\K\chi^\m.
 \end{equation}
  In particular, we obtain
 \begin{equation}\label{e:dimension}
 \dim_{\K} S_{\aa}=|P_\aa\cap M|.
 \end{equation}
For semi-ample degrees $\aa$ the situation is especially nice.  First, $P_{\aa}$ is a lattice polytope. Second, one  can reconstruct $\aa$ from its polytope $P_{\aa}$ by setting $$\aa=[\sum_{j=1}^r a_j D_j]\quad \text{where}\ \ 
 a_j=-\min_{\m\in P_{\aa}\cap M}\la \m,\vv_j\ra.$$ This implies that for $\aa$, $\aa'$ semi-ample we have 
 \begin{equation}\label{e:Minksum} 
P_{\aa+\aa'}=P_{\aa}+P_{\aa'},
 \end{equation}
 where the sum on the right is the Minkowski sum of polytopes.  See \cite[Sec.3.4]{F} or \cite{CLSch} for details.

%Note that the semi-ample degrees form a semigroup which we will denote by $\cl K$.

Next we recall the ideal--scheme correspondence
for simplicial toric varieties (for details, see \cite[Sections 2 and 3]{Coxhom}).
%Every homogeneous ideal $J$ in $S$ defines an ideal sheaf $\tilde J\subset \cO_X$, and
%so gives rise to the closed subscheme $Y\subset X$ with $\cO_Y=\cO_X/\tilde J$. 
%As in the case of $X=\Pp^n$, different homogeneous ideals may define the same subscheme. 
Recall that the saturation of $J$ with respect to the irrelevant ideal $B$ (or simply  the $B$-saturation of $J$) is the ideal
$$J:B^{\infty}:=\{F\in S \ : \ F \cdot B^k \subseteq J \quad \mbox{for some integer} \;\; k\geq 0\}.$$
If $J=J:B^{\infty}$ we say that $J$ is $B$-saturated. 
Furthermore,
$J$ is said to be Pic-saturated if $J_{\aa}=(J:B^{\infty})_{\aa}$ for all $\aa \in \Pic (X)$ and is called Pic-generated if it is generated by polynomials whose degrees belong to $\Pic (X)$. It is clear that a $B$-saturated homogeneous ideal is Pic-saturated, as $\Pic(X) \subseteq \Cl(X)$ and the two notions coincide when $X$ is smooth in which case $\Pic(X) = \Cl(X)$.
%The following statement is the ideal--scheme correspondence for simplicial toric varieties (see 
%Proposition 2.4, Corollary 3.9, and the discussion before the corollary in \cite{Coxhom}).
%\begin{pro}\label{P:ideal-scheme}
%For every closed  subscheme $Y$ of $X$ there is a unique Pic-generated and Pic-saturated  homogeneous ideal $I(Y)$ of $S$  defining $Y$. 
%In particular, if $J$ defines $Y$ and $J:B^{\infty}$ is Pic-generated then $I(Y)=J:B^{\infty}$.
%When $X$ is smooth (or when ) the ideal
%$I(Y)$ is the $B$-saturation $(J:B^{\infty})$. 
%\end{pro}

\begin{pro}\label{P:ideal-scheme} Let $X$ be a complete simplicial toric variety.
\begin{enumerate}
\item For every closed  subscheme $Y$ of $X$ there is a unique Pic-generated and Pic-saturated  homogeneous ideal $I(Y)$ of $S$  defining $Y$. 
\item For every reduced closed  subscheme $Y$ of $X$ there is a unique radical $B$-saturated  homogeneous 
ideal $I(Y)$ of $S$  defining $Y$.
\end{enumerate}
\end{pro}

By \cite[Section 2]{Coxhom} we identify a reduced closed subscheme $Y\subset X$ defined by $I$ with
the closed subset
$$V_X(I):=\{\pi(\x) \in X ~:~f(\x)=0, ~\forall f\in I\},$$ 
the image of the closed subset $V(I)\subset \K^r\setminus V(B)$ under the geometric quotient map.

%Hence, if $J$ is radical, the closed subscheme $Y$ it defines is nothing but the reduced subscheme $V_X(J)$. 
%We prove the converse below for smooth $X$.
%
%\begin{lm} \label{lemrad}Assume that $X$ is smooth. If $J$ is a homogeneous ideal of $S$ which defines a reduced closed subscheme
%$Y\subset X$ then $J:B^{\infty}=I(Y)=\sqrt {J:B^{\infty}}$. In particular, if $J$ is $B$-saturated, then 
%$J=I(Y)=\sqrt J$.
%\end{lm}
%\begin{proof} Since $Y$ is reduced, $J$ and $\sqrt J$ define the same subscheme, i.e. $Y=V_X(J)$. When $X$ is smooth, Proposition \ref{P:ideal-scheme} implies the following 
%$$J:B^{\infty}=I(Y)=\sqrt{J}:B^{\infty} = \sqrt {J:B^{\infty}},$$ 
%which completes the proof.
%\end{proof}

\begin{defi}\label{D:CI} 
We say that a closed subscheme $Y\subset X$ is a {\it complete intersection} if $I(Y)$ is generated by
a regular sequence of homogeneous polynomials $F_1,\dots, F_k\in S$ 
where $k$ is the codimension of $Y$ in $X$. 
%The subscheme $Y$ is said to be a {\it local complete intersection} if its local rings at the points of its support are complete intersection rings.
\end{defi}

\begin{ex} \label{hir} Let $X=\cl H_{\ell}$ be the Hirzebruch surface which will be a running smooth example. 
We identify $M$ and $N$ with $\Z^2$ and consider the fan in $\R^2$ with rays generated by $\vv_1=(1,0)$, $\vv_2=(0,1)$, $\vv_3=(-1,\ell)$,
and $\vv_4=(0,-1)$.
Writing the maps in the standard bases, we obtain the exact sequence $\mathfrak{P}$:
$$\dis \xymatrix{ \mathfrak{P}: 0  \ar[r] & \Z^2 \ar[r]^{\phi} & \Z^4 \ar[r]^{\deg}& \Z^2 \ar[r]& 0},$$  
where $$\phi=\begin{bmatrix}
 1 & 0 & -1& ~~0 \\
0 & 1 & ~~~\ell& -1
\end{bmatrix}^T   \quad  \mbox{ and} \quad \deg=\begin{bmatrix}
1 & -\ell & 1& 0 \\
0 & ~~~1 & 0&1  
 \end{bmatrix}.$$ 
 As $X$ is smooth, $\cA=\Cl (X)=\Pic (X)=\Z^2$. The ring $S= \K[x,y,z,w]$ is graded by $\deg_{\cA}(x)=\deg_{\cA}(z)=(1,0), \deg_{\cA}(y)=(-\ell,1), \deg_{\cA}(w)=(0,1)$, and the semigroup of semi-ample degrees $\cl K=\N^2 \subset \N\bb=\N \{(1,0), (-\ell,1), (0,1)\}$ in this case. 
The group $G$ is parametrized by the columns of the matrix $\deg$, hence, $$G=\{(t_1,t_1^{-\ell}t_2,t_1,t_2) \:|\: t_1,t_2 \in \T\}.$$ We denote by $[x:y:z:w]$ the homogeneous coordinates of points on $\cl H_{\ell}$, using $[x:y:z:w]:=G \cdot (x,y,z,w)$. For instance, $$[0:0:1:1]:=G \cdot (0,0,1,1)=\{(0,0,t_1,t_2) \:|\: t_1,t_2 \in \T\}.$$
Now consider $J=\la xz,yw \ra$. This is a $B$-saturated radical homogeneous ideal of $S$, where $B=\la xy,yz,zw,wx\ra$. 
%  with $\deg(yw)=(-\ell,2) \notin \cl K$. There is no non-zerodivisor of degree $(-\ell,1)\in \N\bb \setminus \cl K$ (??????). 
So, we have $$Y=V_X(J)=\{[0:0:1:1],[0:1:1:0],[1:0:0:1],[1:1:0:0]\}.$$
Note that $V_X(J)$ is a zero-dimensional complete intersection that does not lie in the torus 
$\T^4/G=\{[1:t_1^{\ell}t_2:t_1^{-1}:1] \: | \: t_1,t_2 \in \T\}\cong\T^2$, where the last isomorphism comes from the short exact sequence $\mathfrak{P}^*$.
\end{ex}
We end this section with a theorem which provides us with certain sufficient conditions for a zero-dimensional 
subscheme to be a complete intersection. 
%or equivalently the minimal free resolution of $I(Y)$ is given by the Koszul complex.
\begin{tm} \label{Koszul} Let $X$ be a simplicial toric variety
and $F_i\in S$ be homogeneous polynomials with $\deg_{\cA}(F_i)\in \Pic(X)$, for $1\leq i\leq n$. 
Suppose that $\la F_1,\dots, F_n\ra$ is Pic-saturated and defines a  zero-dimensional subscheme 
$Y\subset X$. Then 
\begin{enumerate}
\item $I(Y)=\la F_1,\dots, F_n \ra$,
\item the minimal free resolution of $I(Y)$ is given by the Koszul complex of $F_1,\dots, F_n$,
\item $Y$ is a complete intersection in $X$.
\end{enumerate}
\end{tm}
\begin{proof} Part (1) follows from \rp{ideal-scheme} since $\la F_1,\dots, F_n \ra$ is Pic-generated and Pic-saturated.
%As $\deg_{\cA}(F_1),\dots,\deg_{\cA}(F_n) \in \Pic(X)$ means $\la F_1,\dots, F_n \ra$ is Pic-generated, it follows that $I(Y)=\la F_1,\dots, F_n \ra$.
Let $\Supp Y=\{p_1,\dots, p_N\}$ be the support of $Y$ 
%{\com Is this the support of Y? It is not assumed reduced. Mesut: Yes, it is the support and I offer to use this notation throughout the paper. As far as I can think of, reducedness does not play a role in this argument} 
and consider the following preimage $$\dis \pi^{-1}(\Supp Y)=\bigcup_{i=1}^{N} \pi^{-1}(p_i).$$
Since $\pi$ is a morphism of algebraic varieties, all the fibers are isomorphic to the algebraic group~$G$. This implies that $\dis \pi^{-1}(Y)$ has dimension $r-n=\dim G$ as a subvariety of $ \K^r \setminus V(B)$. Thus, $\codim I(Y)=n$. Since  $S$ is Cohen-Macaulay, it follows that $\depth I(Y)=n$, which together with \cite[Theorem 17.4]{eis} implies that the Koszul complex of $F_1,\dots, F_n$ gives a minimal free resolution of $I(Y)$. This proves part (2).
Finally, the Koszul complex of $F_1,\dots, F_n$ is exact if and only if $F_1,\dots, F_n$ form a regular sequence in $S$ which completes the proof of part (3).
\end{proof}

\section{Multigraded Hilbert function} \label{S:mhf}

Let $S= \K[x_1,\dots,x_r]$ be the homogeneous coordinate ring of a complete simplicial toric variety $X$ over an algebraically closed field $\K$. In this section, we investigate the behavior of the multigraded Hilbert functions of zero-dimensional closed subschemes $Y$ of $X$. For the rest of the paper, we additionally assume that $\Cl(X)$ has no torsion which is the case, for example, for smooth $X$ by \cite[Proposition 4.2.5]{CLSch}. Under these circumstances, $\dis S=\bigoplus_{\aa \in \cA} S_{\aa}$ is \textit{positively} multigraded by $\cA$. In particular, every $S_\aa$
is finite-dimensional.
%{\com is this the definition? Mesut: The definition is a little bit more general and in the case where S is the coordinate ring of a toric variety, it is called positively multigraded if $S_{\aa}$ is finite dimensional (which follows from the completeness of X) and class group is free} %, where $S_{\aa}$ is the \textit{finite} dimensional vector space of homogeneous polynomials in $S$ of degree $\aa$. 
If $I$ is an $\cA$-graded ideal in $S$, it is called homogeneous, and the quotient ring $S/I$ inherits the multigraded structure yielding a decomposition $\dis S/I=\bigoplus_{\aa \in \cA} ({S/I})_{\aa}$, where $({S/I})_{\aa}=S_{\aa}/I_{\aa}$ is a finite dimensional vector space spanned by monomials of degree $\aa$ which do not belong to $I$. This gives rise to the multigraded Hilbert function and series defined respectively by 
$$H_{S/I}(\aa):=\dim_{\K}(S/I)_{\aa}=\dim_{\K}{S}_{\aa}-\dim_{\K} I_{\aa}$$
$$\mbox{and}\quad HS_{S/I}(\t)=\sum_{\aa \in \cA} H_{S/I}(\aa) \t^{\aa}.$$

%Notice that this might not be true for any finitely generated $\cA$-graded $S$-module. 
Since the grading is positive, the semigroup $\N\bb$ is \textit{pointed} by Corollary 8.8 in \cite{MS}. In this case the ordering $\preceq $ is a partial order, where $\aa \preceq  \aa' $ if $\aa' -\aa \in \N\bb$. By Theorem $8.20$ in \cite{MS}, the Hilbert series is a rational function, that is, we have
$$HS_{S/I}(\t)=\frac{p_{S/I}(\t)}{(1-\t^{\bb_1})\cdots (1-\t^{\bb_r})},$$
for a unique Laurent polynomial $p_{S/I}(\t)$ with integer coefficients. %{\com right? Mesut: Yes
\begin{defi} Let $Y$ be a closed subscheme of $X$ and $I(Y)$ the corresponding 
homogeneous ideal in $S$ as in \rp{ideal-scheme}. We define the {\it multigraded Hilbert function} $H_Y$ of 
$Y$ to be the multigraded Hilbert function of the quotient ring $S/{I(Y)}$.
\end{defi}

As in the classical case the multigraded Hilbert function gives information about the geometry of $Y$ such as the degree of $Y$, denoted $\deg(Y)$, which is defined to be the length of $Y$ as a subscheme. Recall that the length of $Y$ is the dimension of the $\K$-vector space $\Gamma(Y,\cO_Y)$ 
%{\com for 0-dim subschemes? Mesut: Yes, but it should be clear that we are always assuming that Y is zero dimensional within the section}. 
Therefore, if $Y$ has %is a closed subscheme of $X$ of 
degree $N$ and is supported at points $p_1,\dots, p_s$, then $\Gamma(Y,\cO_Y)=\cO_{Y,p_1} \times \cdots \times \cO_{Y,p_s}$ and $N=d_1+\cdots+d_s$, where the positive integer $d_i=\dim_{\K} \cO_{Y,p_i}$ is called the multiplicity of $Y$ at $p_i$ for $i=1,\dots,s$. In particular, $\deg(Y)=|Y|=N$, when $Y$ is a reduced union of $N$ distinct points. 

\subsection{Behavior of the multigraded Hilbert function}
We start with  the classical case when $X=\Pp^{n}$ and so $\cA=\Z$, $\N\bb=\N$ as all $\bb_i=1$.
Assume $Y$ is a reduced zero-dimensional subscheme of $\Pp^{n}$. Then the Hilbert function of $Y$ has the following nice properties
% and the classical Hilbert function of $Y$ has the following nice properties when $Y$ is a reduced subscheme, 
(see for instance \cite{DRT} or \cite{GM}). 
\begin{pro} \label{Pn} Let $Y\subset\Pp^n$ be a reduced zero-dimensional subscheme. 
Let $I(Y)=\bigoplus_{k=k_0}^{\infty} I(Y)_{k}$ with $I(Y)_{k_0}\neq 0$. Then there is an integer $a_Y$ called the $a$-invariant of $Y$ such that the following hold.\\
(\textbf{i}) $ H_Y(k)=\dim_{\K}S_k$ if and only if $k <k_0$, \\
(\textbf{ii}) $ H_Y(k)<H_Y(k+1)< |Y|$ for $0\leq k <a_Y$,  \\
(\textbf{iii}) $ H_Y(k)= |Y|$ for $a_Y<k. \hfill \Box$ 
\end{pro} 

In the rest of the section, we investigate how much these properties of the classical Hilbert function extend to the multigraded setting.  Namely, in Theorem \ref{genHF} below we generalize properties (\textbf{i}) 
and (\textbf{ii}) (under an extra assumption) to zero-dimensional subschemes 
of any simplicial toric variety $X$.  Furthermore, \rt{bound} provides a generalization
of property (\textbf{iii}) for complete intersection subschemes of $X$.  In the important case when $Y$
lies in the torus $\T^n$ the properties of $H_Y$ are summarized in Corollary \ref{HF}.

To start with, let us mention that properties (\textbf{ii})  and (\textbf{iii}) need not hold without extra
assumptions as the following illustrates.
\begin{ex} \label{123} Let $X=P(1,2,3)$ be the weighted projective surface with homogeneous coordinate ring $S=\K [x,y,z]$ where $\deg_{\cA}(x)=1$, $\deg_{\cA}(y)=2$ and $\deg_{\cA}(z)=3$. Consider the ideal $J=\la x,z \ra$ which gives rise to the irreducible reduced complete intersection subscheme $Y=\{[0:1:0]\}$. As $J$ is $B$-saturated, $I(Y)=J$. The multigraded Hilbert function is $H_Y(2\aa)=1$ and $H_Y(2\aa+1)=0$, for all non-negative integers $\aa$, since the Hilbert series is as follows
$$HS_{S/J}(\t)=\frac{1-\t-\t^3+\t^4}{(1-\t)(1-\t^{2})(1-\t^{3})}=\frac{1}{1-\t^{2}}.$$

Consider now the ideal $J'=\la y,z^3 \ra$ which gives rise to the triple point $Y'=\{[1:0:0]\}$. Again $I(Y')=J'$ and the multigraded Hilbert function is given by $H_{Y'}(\aa)=1$, for $\aa=0,1,2$, $H_{Y'}(\aa)=2$ for $\aa=3,4,5$, and $H_{Y'}(\aa)=3=\deg(Y')$, for all integers $\aa \geq 6$, since the Hilbert series is 
$$HS_{S/J'}(\t)=\frac{1-\t^2-\t^9+\t^{11}}{(1-\t)(1-\t^{2})(1-\t^{3})}=\frac{1+\t^{3}+\t^{6}}{1-\t^{}}.$$

\end{ex}

%We generalize these results to arbitrary zero-dimensional closed subscheme $Y$ of a toric variety $X$ satisfying assumptions $(1)-(2)$. 
%If $X=\Pp(w_0,\dots,w_n)$ is the weighted projective space, then by adapting the proof of Proposition $2.1$ in \cite{DRT} one can prove that $\reg(Y)=a_Y+1+\N w$, where $a_Y$ equals the degree of the rational function corresponding to the Hilbert series of $S/I(Y)$ and $\N w$ is the numerical semigroup generated by $w_1,\dots,w_n$.
%PROOF?

The example above indicates that the behavior of the Hilbert series and Hilbert function depends on the smallest possible degree of a non-zerodivisor in the ring $S/I(Y)$. In the first case this number is two and so $H_Y$ reveals a non-decreasing behavior in the even degrees and odd degrees, separately. In the second case this number is one and $H_Y$ is non-decreasing at each degree. The next lemma confirms this observation and its proof is an extension of the proof of Proposition $3.5$ in \cite{Van}.

\begin{lm} \label{lemHF} If $I$ is a homogeneous ideal of $S$ such that there is a non-zerodivisor in $S/I$ of degree $\aa_0 \in \N\bb$, then the following holds.
$$ H_{S/I}(\aa)\leq H_{S/I}(\aa+\aa_0), \mbox{~~ for all~~} \aa \in \N\bb.$$
%(\textbf{ii}) If $ H_{S/I}(\aa)= H_{S/I}(\aa+\aa_0)$ for some $~\aa \in \N\bb$, then $ H_{S/I}(\aa+\aa_0)= H_{S/I}(\aa+2\aa_0)$, i.e. Hilbert function stabilizes in that direction.
\end{lm}
\begin{proof} Let $F_0\in S_{\aa_0}$ be a non-zerodivisor in $S/I$. Then the following complex is exact:
$$\dis \xymatrix{  0  \ar[r] & S/I \ar[r]^{\ov F_0~~~~} & (S/I)(\aa_0) \ar[r]^{}& (S/(I+\la F_0\ra))(\aa_0) \ar[r]& 0},$$
where $(S/I)(\aa_0)$ is a multigraded ring shifted by $\aa_0$, i.e. $((S/I)(\aa_0))_{\aa}=(S/I)_{\aa+\aa_0} $ for all $\aa\in\cA$.
Restricting to a degree $\aa \in \N\bb$, we get the exact sequence of the corresponding vector spaces: 
$$\dis \xymatrix{  0  \ar[r] & (S/I)_{\aa} \ar[r]^{\ov F_0~~~~} & (S/I)_{\aa+\aa_0} \ar[r]^{}& (S/(I+\la F_0\ra))_{\aa+\aa_0} \ar[r]& 0}.$$
The proof follows from the fact that the first map $\ov F_0$ is injective.
% (\textbf{ii}) If $ H_Y(\aa)= H_Y(\aa+\aa_0)$ for some $~\aa \in \N\bb$, then the map $\ov F_0$ is an isomorphism or equivalently $ (S/(I+\la F_0\ra))_{\aa+\aa_0}=0$. This implies $ (S/(I+\la F_0\ra))_{\aa+2\aa_0}=0$ (WHY?).The last equality considered in the following shifted exact sequence completes the proof: 
%$$ \dis \xymatrix{  0  \ar[r] & (S/I)_{\aa+\aa_0} \ar[r]^{\ov F_0~~~~} & (S/I)_{\aa+2\aa_0} \ar[r]^{}& (S/(I+\la F_0\ra))_{\aa+2\aa_0} \ar[r]& 0}.$$
\end{proof}
\begin{rema}
We note that there exists a non-zerodivisor in $S/I$ of degree $\aa \in \cl K$, if $X$ is smooth and $I$ is $B$-saturated, see Remark $2.5$ in \cite{MacS5}. Sometimes $\cl K$ contains all of them as we see by taking $X=\cl H_{\ell}$ and $B$-saturated $J=\la xz,yw\ra$, as in example \ref{hir}. There is no non-zerodivisor in $S/J$ of degree $\aa \in \N \bb \setminus \cl K$, for if $F\in S_{\aa}$ with $\aa \in \N \bb \setminus \cl K$, then $y$ divides $F$, and thus $wF \in J$. Recall that in this case $V_X(J)$ does not lie in $\T^n$. On the other hand, if $Y \subseteq \T^n$ we have more non-zerodivisors, as we record next.
\end{rema}

\begin{lm}\label{lemdiv} Let $Y \subseteq \T^n$ be a reduced subscheme of $X$. 
Then, for any $\aa \in \N\bb$, there exists a non-zerodivisor in $S/I(Y)$ of degree $\aa$.
\end{lm}
\begin{proof} Consider  $\aa=a_1\bb_1+\cdots+a_r\bb_r$, then we prove that the monomial $\x^\a=x_1^{a_1}\dots x_r^{a_r}\in S_{\aa}$ is a non-zerodivisor in $S/I(Y)$. Suppose that $\x^\a F \in I(Y)$ for some polynomial $F\in S$. Then $Y \subseteq V(x_1)\cup \cdots \cup V(x_r) \cup V(F)$. Since $Y \subset \T^n$, it follows that $Y \subseteq V(F)$ which implies that $F \in I(Y)$, as $I(Y)$ is a radical ideal.
\end{proof}
The next result summarizes general properties of the multigraded Hilbert function in the most general setup.
\begin{tm} \label{genHF}  Let $Y\subset X$ be a reduced zero-dimensional subscheme. 
The multigraded Hilbert function $H_Y$ has the following properties.\\
(\textbf{i}) If $P_{\aa}$ does not contain any $P_{\aa_i}$, for degrees $\aa_i$ of minimal generators of $I(Y)$,  then $H_Y(\aa)=| P_{\aa} \cap M |$,\\
(\textbf{ii}) If there is a non-zerodivisor in $S/I(Y)$ of degree $\bb_j$, for each $~j \in [r]$, then 
$H_Y(\aa) \leq H_Y(\aa')$ for all $\aa \preceq  \aa'$. \\
(\textbf{iii}) $H_Y(\aa) \leq \deg(Y)$, for all $~\aa \in \N\bb.$ 
\end{tm}
\begin{proof}  (\textbf{i}) Since $\dim_{\K}S_{\aa}=|P_{\aa}\cap M|$ by equation \re{dimension}, it is enough to show that $I(Y) \cap S_{\aa}=\{0\}$ as in this case we have $\dim_{\K} I(Y)_{\aa}=0$ and hence
$$H_{Y}(\aa)=\dim_{\K}{S}_{\aa}-\dim_{\K} I(Y)_{\aa}=\dim_{\K}S_{\aa}.$$

\noindent Let $I(Y)=\la F_1, \dots ,F_m\ra$, where $\deg_{\cA}(F_i)=\aa_i$, for all $i \in [m]$. If $P_{\aa}$ does not contain any $P_{\aa_i}$, then $\aa-\aa_i \notin \N\bb$, as otherwise there would be some $\gamma_i \in \N\bb$ such that $\aa=\aa_i + \gamma_i$ which would imply that $P_{\aa} \supseteq P_{\aa_i}+P_{\gamma_i} \supseteq P_{\aa_i}$. Therefore $S_{\aa-\aa_i}=\{0\}$. Now, take any element $F \in  I(Y) \cap S_{\aa}$. Since $F=\sum_{i=1}^{m} G_i F_i$, with $G_i \in S_{\aa-\aa_i}$ it follows that $F=0$, completing the proof.

(\textbf{ii}) If $\aa \preceq  \aa'$ then $\aa'-\aa \in \N\bb$, i.e. there are non-negative integers $\mu_j$ so that 
$\aa'=\aa+\mu_1\bb_1+\cdots+\mu_n\bb_n$ and the result follows from Lemma \ref{lemHF}. 

(\textbf{iii}) This inequality follows easily from the proof of \cite[Proposition 6.7]{MacS4}.
\end{proof}

\begin{rema} The first item in Theorem \ref{genHF} above generalizes the first property of Proposition \ref{Pn}. 
Indeed, note that the polytope $P_k$ corresponding to degree $k$ polynomials is the simplex with vertices $\{0,ke_1,\dots, ke_n\}$.
Therefore $k<k_0\leq k_i$ is equivalent to $P_k \subset P_{k_0}\subseteq P_{k_i}$, where $k_0$ is the least degree among the degrees $k_i$ of minimal generators of $I(Y)$.
\end{rema}

\begin{rema} If the toric variety is a product of projective spaces $X=\Pp^{n_1} \times \cdots \times \Pp^{n_r}$ and $Y=d_1p_1+\cdots+d_sp_s$ is the closed subscheme determined by the ideal $J=I(p_1)^{d_1} \cap \dots \cap I(p_s)^{d_s}$, we have $\cA=\Z^r$ and $\N\bb=\N^r$, as the first $n_1$ variables have $\cA$-degree $\bb_1=e_1$, and the second $n_2$ variables  have  $\cA$-degree $\bb_2=e_2$ and so on. 
In this case Sidman and Van Tuyl proved stronger nondecreasing and stabilization properties of the 
Hilbert function of  $Y$. In particular, they showed that $ H_Y(\aa)\leq H_Y(\aa+e_j)$ for each $j \in [r]$ and 
all $\aa \in \N^r$. Moreover, if $ H_Y(\aa)= H_Y(\aa+e_j)$ for some $j \in [r]$ and for some $\aa \in \N^r$, then $ H_Y(\aa+e_j)= H_Y(\aa+2e_j)$, i.e. the Hilbert function stabilizes in that direction
(see \cite[Proposition 1.9]{Sid}).

%\begin{pro} With the notations above the following hold.\\
%(\textbf{i}) $ H_Y(\aa)\leq H_Y(\aa+e_j)$ for each $~j \in [r]$ and for all $~\aa \in \N^r$,\\
%(\textbf{ii}) If $ H_Y(\aa)= H_Y(\aa+e_j)$ for some $~j \in [r]$ and for some $~\aa \in \N^r$, then $ H_Y(\aa+e_j)= H_Y(\aa+2e_j)$, i.e. Hilbert function stabilizes in that direction,  \\
%(\textbf{iii}) $ H_Y(\aa) \leq \deg(Y)$ for all $~\aa \in \N^r. \hfill \Box$ 
%\end{pro} 
\end{rema}

By the virtue of \cite[Proposition 6.7]{MacS4}, and to avoid unnecessary technicalities, we make the following

\begin{defi}  The {\it multigraded regularity} of $Y$, denoted $\reg(Y)$, is the set of $\aa\in \N\bb$ for which $H_Y(\aa)=\deg(Y)$, where $Y$ is any zero-dimensional closed subscheme of $X$. 
\end{defi}

Multigraded regularity $\reg(Y)$ is an interesting invariant measuring complexity of the subscheme $Y$. 
As an application, we use a bound on $\reg(Y)$
 to eliminate trivial toric (complete intersection) codes, see \rs{codes}.

\begin{rema} In the classical case of the projective space, Proposition \ref{Pn} yields $\reg(Y)=1+a_{Y}+\N$. On the other hand, by Proposition $2.1$ in \cite{DRT}, the $a$-invariant $a_Y$ equals the degree of the rational function corresponding to the Hilbert series of $S/I(Y)$, that is we have $$a_Y=\deg(p_{S/I(Y)})-(n+1).$$ 
\end{rema}

If we recall Example \ref{123}, we see that the Hilbert series of $S/I(Y)$ has degree $-2$ as a rational function, but $H_Y(\aa)$ does not stabilize after $-2$, i.e. for $\aa \geq-1$. Indeed $\reg(Y)=2\N\neq -1+\N$. On the other hand, the Hilbert series of $S/I(Y')$ has degree $5$ as a rational function and $H_{Y'}(\aa)$ stabilizes after $5$, i.e. $\reg(Y')=6+\N$. The following generalizes this to weighted projective spaces under a mild condition which holds true in the most interesting case of subschemes lying inside the torus if at least one weight is trivial, e.g. $w_1=1$, see Lemma \ref{lemdiv}.

\begin{pro} \label{wps} Let $X=P(w_1,\dots,w_{n+1})$ be a weighted projective space and $Y$ be a subscheme such that $S/I(Y)$ has a non-zerodivisor of degree $1$. Then, there is an integer $a_Y$ satisfying $\reg(Y)=1+a_Y+\N$. Moreover, $a_Y$ equals the degree of the rational function corresponding to the Hilbert series of $S/I(Y)$, that is,
$$a_Y=\deg(p_{S/I(Y)})-\bb_1-\cdots-\bb_{n+1}.$$ 
\end{pro}
\begin{proof} By \cite[Proposition 4.4]{MacS4}, the set $\reg(Y)\subseteq \N$ is not empty. Let $a_Y$ be the integer such that $1+a_Y$ is the smallest element in this set. The assumption together with Lemma \ref{lemHF} and Theorem \ref{genHF} {\bf(iii)}  implies that $H_Y(\aa)\leq H_Y(1+\aa) \leq \deg(Y)$ for all $\alpha\in\N$.
 Since $H_Y(1+a_Y)=\deg(Y)$, the first claim follows by taking $\aa=1+a_Y$ to start with.

 The second part can be done by adopting carefully the proof of Proposition $2.1$ in \cite{DRT} and replacing $|Y|$ with $\deg(Y)$.
\end{proof}

\subsection{Multigraded Hilbert functions of complete intersections} 
In the rest of the section $Y$ is a complete intersection subscheme of a simplicial toric variety $X$ as in \rd{CI} which is cut out by hypersurfaces of degrees $\aa_1, \dots,\aa_n$. In this case, we have the following nice combinatorial formula for $H_Y(\aa)$.

\begin{pro}\label{P:HilbFla}  Let $Y$ be a complete intersection subscheme of $X$ such that $I(Y)$ is generated by $n$ homogeneous polynomials with degrees $\aa_1,\dots,\aa_n$. 
Then for any $\aa\in\cA$ we have
$$
H_Y(\aa)=\sum_{s=0}^{n}\,\, (-1)^s\!\!\! \sum_{I\subseteq [n], |I|=s} |P_{\aa-\aa_I} \cap M |,
$$
where $\aa_I=\sum_{i\in I}\aa_i$.
\end{pro}

\begin{proof}
Since $I(Y)$ is a complete intersection, its minimal free resolution is given by the Koszul complex so that we have the following exact sequence 
$$\dis { 0  \rightarrow W_n \rightarrow \cdots \rightarrow  W_s  \rightarrow \cdots \rightarrow  W_1 \rightarrow S_{\aa}\rightarrow (S/I(Y))_{\aa}\rightarrow 0 },$$
where, for every $s=1,\dots,n$, the vector space $W_s$ is given by 
$$W_s=\bigoplus_{I\subseteq [n], |I|=s} S(-\aa_{I})_{\aa}=\bigoplus_{I\subseteq [n], |I|=s} S_{\aa-\aa_{I}}.$$
Therefore, combining this with \re{dimension} we obtain:
 \begin{eqnarray}\label{e:HF} 
 H_Y(\aa)&=&\dim_{\K}S_{\aa}+\sum_{s=1}^{n} (-1)^s \dim_{\K} W_s\nonumber\\
&=& \sum_{s=0}^{n} (-1)^s \sum_{I\subseteq [n], |I|=s} |P_{\aa-\aa_I} \cap M |.
 \end{eqnarray}
\end{proof}

%In the classical case of the projective space, $H_Y(\aa)=|Y|$ if and only if $\aa>a_Y$. By Proposition $2.1$ in \cite{DRT}, the $a$-invariant $a_Y$ equals the degree of the rational function corresponding to the Hilbert series of $S/I(Y)$, that is we have $$a_Y=\deg(p_{S/I(Y)})-(n+1).$$ The polynomial $p_{S/I(Y)}$ is determined by the multigraded Betti numbers and Betti degrees by Proposition $8.23$ in \cite{MS}.
%When $Y$ is a complete intersection cut out by hypersurfaces of degrees $\aa_1, \dots,\aa_n$,
%the minimal free resolution is given by the Koszul complex. In this case the
% $a$-invariant agrees with what is called the {\it critical degree} $a_Y=\aa_1+\cdots+\aa_n-(n+1)$. 
% For the details see the proof of Corollary $2.6$ in \cite{DRT}.

\begin{rema} When $Y$ is a complete intersection as above, by Proposition 8.23 in \cite{MS}, we have $p_{S/I(Y)}=\sum_{s=0}^{n} (-1)^s \sum_{I\subseteq [n], |I|=s} \t^{\aa_I}$, where $\aa_I=\sum_{i\in I}\aa_i$. Hence, if $X=\Pp^n$, then $p_{S/I(Y)}$ has degree $\aa_1+\cdots+\aa_n$, and thus $a_Y$
coincides with what is called the {\it critical degree} $a_Y=\aa_1+\cdots+\aa_n-(n+1)$. This formula can be extended to the case of weighted projective spaces as $a_Y=\aa_1+\cdots+\aa_n-\bb_1-\cdots-\bb_{n+1}$ whenever the assumptions of  Proposition \ref{wps} are satisfied. The question of whether $\aa_1+\cdots+\aa_n-\bb_1-\cdots-\bb_{r}+\N\bb \subseteq \reg(Y)$, for any simplicial toric variety $X$ has an immediate negative answer provided by the following example as the critical degree $\aa_1+\aa_2-\bb_1-\cdots-\bb_4=(4,0)$ does not give a lower bound for the multigraded regularity. 
\end{rema}
\begin{ex} \label{critical} Let $X=\cl H_{2}$ be the Hirzebruch surface and $J=\la F_1, F_2 \ra$, where $F_1=x^3z-x^2z^2-2xz^3+z^4$ and $F_2=x^{4}y^2-w^2$. It is easy to see that $J$ is $B$-saturated and defines a zero-dimensional  complete intersection not lying in $\T^2$, 
$V_X(J)=\{[1:1:0:\pm1],[1:1:\zeta_1:\pm1],[1:1:\zeta_2:\pm1],[1:1:\zeta_3:\pm1]\},$
 where $\zeta_i$ are three distinct solutions of $\zeta^3-2\zeta^2-\zeta+1=0$ in $\K$. One can see from the following matrix that $(3,1)+\N \bb=\reg(Y)$, where the value of the Hilbert function at the origin is in red and $\reg(Y)$ is in blue.
 $$\small\begin{bmatrix}
      0\  \  0\  \  0\  \  0\  \  0\  \  0\  \  1\  \  2\  \  4\  \ 6\  \ 7\  \  \textcolor{blue}{8\  \ 8\  \ 8\  \ 8\  \ 8\  \ 8\  \ 8\  \ 8\  \ 8\  \ 8}\\
      0\  \  0\  \  0\  \  0\  \  0\  \  0\  \  0\  \  0\  \  1\  \  2\  \  4\  \ 6\  \ 7\  \textcolor{blue}{\ 8\  \ 8\  \ 8\  \ 8\  \ 8\  \ 8\  \ 8\  \ 8}\\
      0\  \  0\  \  0\  \  0\  \  0\  \  0\  \  0\  \  0\  \  0\  \  0\  \  \textcolor{red}{1}\  \  2\  \ 3\  \ 4\  \ 4\  \ 4\  \ 4\  \ 4\  \ 4\  \ 4\  \ 4
 \end{bmatrix}.$$
\end{ex}
Instead, we have the following lower bound for the multigraded regularity of complete intersections of semi-ample hypersurfaces.

\begin{tm} \label{T:bound} 
Let $Y$ be a complete intersection subscheme of $X$ such that $I(Y)$ is generated by $n$ homogeneous polynomials with semi-ample degrees $\aa_1,\dots,\aa_n$. If there is a non-zerodivisor in $S/I(Y)$ of degree $\bb_j$, for each $~j \in [r]$, then
 $$\aa_1+\cdots+\aa_n+\N\bb \subseteq \reg(Y).$$
\end{tm}
\begin{proof} Let $F_1,\dots,F_n$ be the generators of $I(Y)$. Their zero loci define divisors
$E_1,\dots,E_n$ on $X$.
Note that $\deg Y$ equals the intersection number $(E_1,\dots,E_n)$. On the other hand,
each $E_i$ is linearly equivalent to a $\T$-invariant divisor $H_i$ of degree $\aa_i$.
By \cite[Sec. 5.5]{F} we have
$$(E_1,\dots,E_n)=(H_1,\dots,H_n)=n!V(P_{\aa_1},\dots,P_{\aa_n}).$$
Therefore, it is sufficient to prove that $H_{Y}(\aa)=n!V(P_{\aa_1},\dots,P_{\aa_n})$, whenever the sum 
$ \aa_1+\cdots+\aa_n \preceq \aa$. We handle this  in two steps. 

Case (1): $\aa_{n+1}:=\aa - (\aa_1+\cdots+\aa_n) $ is semi-ample. 

For notational convenience, set $P_i=P_{\aa_i}$, for $i=1,\dots,n+1$. As before let 
$P_I=\sum_{i\in I}P_{i}$ and $\aa_I=\sum_{i\in I}\aa_{i}$ for a subset $I\subseteq [n]$.

By \cite[Proposition 9.8]{post}, we have the following formula for the $n$-dimensional volume of $z_1P_1+\cdots+z_{n+1}P_{n+1}$
$$\Vol_n(z_1P_1+\cdots+z_{n+1}P_{n+1})=\sum_{(i_1,\dots,i_n)\in [n+1]^n} V(P_{i_1},\dots,P_{i_n})z_{i_1}\cdots z_{i_n},$$
where $V(P_{i_1},\dots,P_{i_n})$ is the mixed volume of the polytopes $P_{i_1},\dots,P_{i_n}$.
Furthermore, let $\Ehr(z_1,\dots,z_{n+1})$ be the Ehrhart polynomial that counts the number of lattice points in the Minkowski sum 
$z_1P_1+\cdots+z_{n+1}P_{n+1}$ for non-negative integers $z_1,\dots, z_{n+1}$. 
Asymptotically, as $z_i\to\infty$, the number of lattice points equals the volume, so we have
$$\Ehr(z_1,\dots,z_{n+1})=\Vol_n(z_1P_1+\cdots+z_{n+1}P_{n+1})+\mbox{lower order terms in}\; z_1,\dots,z_{n+1}.$$ 

Now consider the functions $V(z_1,\dots,z_n)=\Vol_n(z_1P_1+\cdots+z_nP_n+P_{n+1})$ and $E(z_1,\dots,z_n)=\Ehr(z_1,\dots,z_n,1)$.
By above the coefficient of $z_1\cdots z_n$ in $V$ is $n!V(P_1,\dots,P_n)$. Thus, we obtain
\begin{eqnarray}\label{e:MV1} 
\p^n V/\p z_1\cdots\p z_n=n!V(P_1,\dots,P_n).
\end{eqnarray}
On the other hand,  since $E$ is a polynomial of degree $n$ its partial derivative coincides with its discrete derivative when the order is $n$, i.e. $\p^n E/\p z_1\cdots\p z_n=\D^n E/ \D z_1\cdots\D z_n$ at $(z_1,\dots,z_n)=(0,\dots,0)$ which is nothing but
\begin{eqnarray}\label{e:MV} 
\frac{\D^nE}{\D z_1\cdots\D z_n}(0,\dots,0)=\sum_{s=0}^{n} (-1)^s \sum_{I\subseteq [n], |I|=s} |(P_{\aa}-P_{I})\cap M|.
\end{eqnarray}
Note that the polytope $P_{\aa}-P_{I}$ is well defined, as by the hypothesis $\aa_1,\dots,\aa_{n+1}$ are all semi-ample and thus $P_{\aa}=P_1+\cdots+P_n+P_{n+1}$, by equation \re{Minksum}. In fact, it follows that $P_{\aa}-P_{I}=P_{\aa-\aa_{I}}$. 

Since we also have $\p^n E/\p z_1\cdots\p z_n=\p^n V/\p z_1\cdots\p z_n$, the combination of \re{MV1} and \re{MV} provides
$$
n!V(P_1,\dots,P_n)=\sum_{s=0}^{n} (-1)^s \sum_{I\subseteq [n], |I|=s} |P_{\aa-\aa_{I}}\cap M|.
$$

Now the statement $H_Y(\aa)=n!V(P_1,\dots,P_n)$, whenever $\aa_1+\cdots+\aa_n \preceq \aa$, follows from \rp{HilbFla}.

In particular,  $H_Y(\aa_1+\dots+\aa_n)=n!V(P_1,\dots,P_n)=\deg Y.$

%WE NEED TO EXPLAIN WHY $n!V(P_1,\dots,P_n)=\deg Y$. 

Case (2): $\aa_{n+1}:=\aa - (\aa_1+\cdots+\aa_n)$ is not semi-ample. 

We claim that there is a semi-ample degree $\aa'_{n+1}$ with $\aa_{n+1} \preceq  \aa'_{n+1}$. 
Indeed, note that $\cl K$  corresponds to the set of the lattice points of a full-dimensional cone $C_{\cl K}$ in $\cA\otimes\R\cong\R^{r-n}$.
Since $C_{\cl K}\cap (\aa_{n+1}+C_{\cl K})$ is unbounded, it must contain a lattice point $\aa_{n+1}'\in \cl K\cap (\aa_{n+1}+\cl K)$. (To see that $C_{\cl K}\cap (\aa_{n+1}+C_{\cl K})$ is unbounded, one
can take any ray pointing to the interior of $C_{\cl K}$ and show that it must eventually intersect the shifted cone  $\aa_{n+1}+C_{\cl K}$.)

Now $\aa':=\aa_1+\cdots+\aa_n+\aa'_{n+1}$ is semi-ample and satisfies $\aa\preceq  \aa'$. Thus, by Case (1) and Theorem \ref{genHF} (ii), we have 
$$\deg Y=H_Y(\aa_1+\cdots+\aa_n) \leq H_Y(\aa) \leq H_Y(\aa')=\deg Y,$$
which completes the proof.
\end{proof}
\begin{rema} When $X$ is smooth and $Y$ is a complete intersection of semi-ample hypersurfaces of degree $\aa_1,\dots,\aa_n$, the lower bound for $\reg(Y)$ provided by the proof of Proposition $2.10$ in \cite{MacS5} is $\aa_1+\cdots +\aa_n+\cl K$. Thus, our bound improves this even for smooth toric varieties.
\end{rema}

Let us finish the section by summarizing the previous results in the special case of $Y\subseteq \T^n$.
\begin{coro} \label{HF} If $Y\subseteq \T^n$ is reduced, then the following hold.\\
(\textbf{i}) If $\aa\not\succeq \aa_i$, for degrees $\aa_i$ of minimal generators of $I(Y)$, then
$$H_Y(\aa)=|P_{\aa}\cap M|.$$\\
(\textbf{ii}) $H_Y(\aa)\leq H_Y(\aa'), \text{ for all }\aa \preceq \aa'$, $\aa \in\N\beta$.\\
(\textbf{iii}) If ~$Y$ is a complete intersection subscheme of $X$ such that $I(Y)$ is generated by homogeneous polynomials with semi-ample degrees $\aa_1,\dots,\aa_n$, then 
%for any 
%$\aa \succeq \aa_1+\cdots+\aa_n$ we have 
$$H_Y(\aa)=|Y|,\text{ for all } \aa \succeq \aa_1+\cdots+\aa_n.$$
% $$\aa_1+\cdots+\aa_n+\N\bb \subseteq \reg(Y).$$
\end{coro}
\begin{proof} The first claim is just Theorem \ref{genHF}, part (i).
By Lemma \ref{lemdiv}, there exists a non-zerodivisor in $S/I(Y)$ of degree $\aa$, 
for any $\aa \in \N\bb$. So, the second claim follows from Theorem \ref{genHF} (ii) and the third 
follows from \rt{bound}.
\end{proof}

\section{Evaluation codes on complete intersections} \label{S:codes} 
In this section we apply our results in \rs{mhf} to dimension calculation for
evaluation codes on complete intersections in a toric variety. 

Recall the basic definitions from coding theory. Let $\F_q$ be a finite field of $q$ elements and 
$\F_q^*=\F_q\setminus\{0\}$ its multiplicative group. 
A subspace $\cC$ of $ \F_q^{N}$ is called a \textit{linear code}, and its elements 
${c}=(c_{1},\dotsc,c_{N})$ are called \textit{codewords}.  The number $N$ is called the {\it block-length} of $\cC$.
The {\it weight} of $c$ in $\cC$ is the number of non-zero entries in $c$.
The {\it distance} between two codewords $a$ and $b$ in $\mathcal{C}$ is the weight of $a-b\in\cC$.
The minimum distance between distinct codewords in $\cC$ is the same as the minimum weight of  non-zero codewords in $\cC$. The block-length $N$, the dimension $k=\dim_{\F_q}(\cC)$, and the minimum
distance $d=d(\cC)$ are the basic parameters of $\cC$. 

Let $\K=\overline{\F}_q$ 
be the algebraic closure of $\F_q$. We let $\Phi\in\Gal(\K/\F_q)$ denote the Frobenius  automorphism, so
$\Phi(x)=x^q$ for all $x\in\K$ and the restriction of $\Phi$ to $\F_{q^k}$ generates the cyclic group
$\Gal(\F_{q^k}/\F_q)$ for every $k\in\N$.

Now, let $X$ be a simplicial complete toric variety over $\K$ of dimension $n$ with torsion-free class group and $S$ its 
homogeneous coordinate ring as in the previous sections. Recall that $S_\aa$ 
is spanned by characters of the torus $\T^n$ as in \re{span}. 
The group  $\Gal(\K/\F_q)$ acting on the coefficients in the direct sum defines an action on $S_\aa$.
Thus the subset $S_\aa^\Phi$ of $S_\aa$ of elements  invariant
under the Frobenius automorphism consists of the $\F_q$-linear combinations of the characters:
\begin{equation}\label{e:Frob}
S_\aa^\Phi\cong \bigoplus_{\m\in P_{\aa}\cap M}\F_q\chi^\m.
\end{equation}
Note that the dimension of the $\F_q$-vector space $S_\aa^\Phi$ and the dimension
of the $\K$-vector space $S_\aa$ are the same.
 
Let $Y=\{p_1,\dots, p_N\}$ be a zero dimensional reduced subscheme of $X$,
contained in $(\F_q^*)^n$. Fix a degree $\aa\in\N\beta$ and a monomial $F_0=\x^{\phi(\m_0)+\a} \in S_{\aa}$,
where  $\m_0\in M$, $\a$ is any element of $\Z^r$ with $\deg(\a)=\aa$, and $\phi$ as in the exact sequence $\mathfrak{P}$.
This defines the {\it evaluation map}
\begin{equation}\label{e:evalmap}
\ev_{Y}:S^\Phi_\aa\to \F_q^N,\quad F\mapsto \left(\frac{F(p_1)}{F_0(p_1)},\dots,\frac{F(p_N)}{F_0(p_N)}\right).
\end{equation}
The image $\cC_{\aa,Y}=\text{ev}_{Y}(S^\Phi_\aa)$ is a linear code, called {\it evaluation code} on $Y$. 
It can be readily seen that different choices of $F_0\in S_\aa$ yield to equivalent codes. Clearly, the block-length of $\cC_{\aa,Y}$ equals $N=|Y|$.

In the special case of $Y=(\F_q^*)^n$ this construction produces {\it toric codes}
which was introduced for the first time by Hansen in \cite{Ha0, Ha1}.  A way to compute the dimension of a toric code is given in \cite{Ru}. For general information about algebraic geometry codes we refer to \cite{TV} and 
\cite{Li}.

\begin{rema}\label{R:two-defs} We remark that the toric complete intersection code $\cC_{\cL(P), Y}$
 studied in \cite{sop} is also a particular case of the above evaluation code. Indeed, in \cite{sop} the code $\cC_{\cL(P), Y}$  is defined as the image of
$$\dis \text{ev}_{Y}:\cL(P) \to \F_q^N,\quad f \mapsto  (f(p_1),\dots,f(p_N)),$$
where $\cL(P)$ is the subspace of the Laurent polynomial ring $\F_q[t_1^{\pm 1},\dots,t_n^{\pm 1}]$ spanned by the monomials corresponding to the lattice points of $P$, and $Y=\{p_1,\dots, p_N\}\subset (\F_q^*)^n$
is the set of common zeroes of a Laurent polynomial system $f_1=\dots=f_n=0$ which
meets the Bernstein--Kushnirenko bound.  There is a standard procedure of homogenizing a Laurent polynomial $f$ with respect to its Newton polytope $P$, which produces a homogeneous polynomial $F\in S$ of semi-ample degree $\aa$ with $P_\aa=P$ (see  \cite[p. 124]{CaD}). In this case, $f=F/F_0$ for
a certain monomial $F_0\in S_\aa$. It is not hard to show that under the conditions imposed on $Y$
in \cite{sop}, the homogenizations $F_1,\dots, F_n$ of the $f_i$ generate an ideal $I\subset S$
defining $Y$, i.e.  $V_X(I)=Y$. On the other hand,  for any semi-ample degree $\aa$ the space $\cL(P_\aa)$
is identified with $S^\Phi_\aa$, by \re{Frob}. Therefore, the toric complete intersection code 
$\cC_{\cL(P), Y}$ coincides with the evaluation code $\cC_{\aa,Y}$ as defined above.
\end{rema}

The advantage of the above observation is that when $Y \subseteq (\F_q^*)^n$ it helps us to see why there is no harm in assuming without loss of generality that $X$ is smooth. This is because when $X$ is singular, one can resolve its singularities by refining the fan which changes the homogeneous coordinate ring $S$ and thereby $S_\aa$ but does not effect the code because $\cL(P_\aa)$ does not change. 

%Finally, the vector space $S_\aa^\Phi$ has a basis consisting of monomials $\{\x^{\a_1},\dots,\x^{\a_k} \}$ and thus, the code $\cC_{\aa,Y}$ has a $k \times N$ generating matrix whose rows are obtained by evaluating these monomials at $N$ points of $Y \subseteq (\F_q^*)^n$, and whose rowspace over $\F_q$ is the code $\cC_{\aa,Y}$. 

%The block-length of $\cC_{\aa,Y}$ equals $N=|Y|$, and the minimum distance was studied in
 %\cite{sop}. 
 The next proposition provides a way to calculate the dimension of the code as the
 value of the Hilbert function $H_Y(\aa)$. 
 %{\com Why are we considering Hilbert function over the algebraic closure? Why not over $\F_q$?}
 
  \begin{pro}\label{P:CodeHF}  Let $Y \subseteq (\F_q^*)^n$ be a reduced closed subscheme of $X$. Then, for any $\aa\in\cA$ the dimension of
$\cC_{\aa,Y}$ equals $H_Y(\aa)$.
\end{pro}

\begin{proof} Let $\ev_Y: S_\aa\to\K^N$ be the $\K$-linear evaluation map extending 
the one in~\re{evalmap}. Since
$I(Y)$ is radical, we have $\ker(\ev_{Y})= I(Y)_{\aa}$, which yields
$$\dim_{\K}(\ev_{Y}(S_\aa))=\dim_{\K}(S_\aa)-\dim_{\K} \ker(\text{ev}_{Y})=H_Y(\aa).$$

On the other hand, $\dim_{\K}(\ev_{Y}(S_\aa))=\dim_{\F_q}(\ev_{Y}(S_\aa^\Phi))$ as
both maps are represented by the same matrix with entries in $\F_q$
in the monomial basis corresponding to the lattice points of $P_\aa$.
By definition $\dim_{\F_q}(\ev_{Y}(S_\aa^\Phi))$ is the dimension of the code $\cC_{\aa,Y}$,
which completes the proof of the proposition. 
\end{proof}
The behavior of the multigraded Hilbert function gives information about equivalence of the corresponding codes as we discuss now.
\begin{pro} \label{finitecodes} Let $Y \subseteq (\F_q^*)^n$ be a reduced closed subscheme of $X$ as before. If $H_Y(\aa)=H_Y(\aa+\aa_0)$ then the codes $\cC_{\aa,Y}$ and $\cC_{\aa+\aa_0,Y}$ are equivalent. Therefore, there are only finitely many non-equivalent codes $\cC_{\aa,Y}$.
\end{pro}
\begin{proof} As $Y \subseteq (\F_q^*)^n$, the monomial $\x^{\a_0}$ of degree $\aa_0$ is a non-zerodivisor by the proof of Lemma \ref{lemdiv}. Hence, we have the following exact sequence
 $$\dis \xymatrix{  0  \ar[r] & (S/I(Y))_{\aa} \ar[r]^{\ov{\x^{\a_0}}~~~~} & (S/I(Y))_{\aa+\aa_0} \ar[r]^{}& (S/(I+\la \x^{\a_0}\ra))_{\aa+\aa_0} \ar[r]& 0}.$$
When $H_Y(\aa)=H_Y(\aa+\aa_0)$, the injection above becomes an isomorphism and we have 
$$(S/I(Y))_{\aa+\aa_0}=\ov{\x^{\a_0}}\cdot(S/I(Y))_{\aa} .$$

On the other hand,  since the kernel of the evaluation map $\ev_Y: S_\aa\to\K^N$ equals
$I(Y)_{\aa}$ we may identify $(S/I(Y))_{\aa}$ with the code $\cC_{\aa,Y}$. 
%rewrite the evaluation map $\ev_Y: S_\aa\to\K^N$ replacing $S$ with the quotient ring $S/I(Y)$. Since $\ker(\ev_{Y})= I(Y)_{\aa}$, the new map will be an isomorphism from $(S/I(Y))_{\aa}$ onto the code $\cC_{\aa,Y}$. 
Let $\x^{\a_1},\dots,\x^{\a_k}$
be monomials in $S_\aa$ whose classes form a basis for  $(S/I(Y))_{\aa}$ and, hence, whose images under the evaluation map form a basis for $\cC_{\aa,Y}$. By above, the monomials $\x^{\a_0}\x^{\a_1},\dots,\x^{\a_0}\x^{\a_k}$
will yield a basis for $\cC_{\aa+\aa_0,Y}$. In these bases the generating matrix of $\cC_{\aa+\aa_0,Y}$
equals the product of a diagonal matrix (consisting of the values of $\x^{\a_0}$ at the points of $Y$) and
the generating matrix of $\cC_{\aa,Y}$.
Therefore, the two codes are monomially equivalent.

As for the last assertion, we first recall that $H_Y$ is non-decreasing in the direction of $\bb_j$ for each $j\in [r]$. Second, it is bounded above by the number of points in $Y$, so it must eventually become constant, which completes the proof. 
\end{proof}
The following shows how this observation can be turned into a powerful method to produce good codes.
\begin{ex} \label{hirci} Let $X=\cl H_{2}$ be the Hirzebruch surface as in example \ref{hir} over the field $\F_5$. Consider the homogeneous ideal $J=\la F_1, F_2 \ra$, where $F_1=x^{2}-z^{2}$ and $F_2=z^{8}y^4-w^{4}$. One can check that $J$ is $B$-saturated and radical, so $Y=V_X(J)$ is the following reduced union of $8$ points lying inside the torus $(\F_5^{*})^2$:
\begin{eqnarray} &[1:1:1:1],[1:1:1:2],[1:1:1:3],[1:1:1:4],&\nonumber \\
&[1:1:4:1],[1:1:4:2],[1:1:4:3],[1:1:4:4].& \nonumber
\end{eqnarray}

 In this case, the toric complete intersection codes have the following table for their dimensions, where the value of the Hilbert function at the origin is in red, recorded to be $1$ at the bottom line, and the trivial codes correspond to the degrees in $(1,3)+\N \bb$ which is in blue,
  $$\small\begin{bmatrix}
    
      0\  \  0\  \  1\  \  2\  \  3\  \  4\  \  5\  \  6\  \  7\  \  \textcolor{blue}{8\  \ 8\  \ 8\  \ 8\  \ 8\  \ 8\  \ 8\  \ 8\  \ 8\  \ 8\  \ 8\  \ 8}\\
      0\  \  0\  \  0\  \  0\  \  1\  \  2\  \  3\  \  4\  \  5\  \  6\  \  7\  \  \textcolor{blue}{8\  \ 8\  \ 8\  \ 8\  \ 8\  \ 8\  \ 8\  \ 8\  \ 8\  \ 8}\\
      0\  \  0\  \  0\  \  0\  \  0\  \  0\  \  1\  \  2\  \  3\  \  4\  \  5\  \  6\  \  6\  \  6\  \ 6\  \ 6\  \ 6\  \ 6\  \ 6\  \ 6\  \ 6\\
      0\  \  0\  \  0\  \  0\  \  0\  \  0\  \  0\  \  0\  \  1\  \  2\  \ 3\  \ 4\  \ 4\  \ 4\  \ 4\  \ 4\  \ 4\  \ 4\  \ 4\  \ 4\  \ 4 \\
      0\  \  0\  \  0\  \  0\  \  0\  \  0\  \  0\  \  0\  \  0\  \  0\  \  \textcolor{red}{1}\  \  2\  \  2\  \ 2\  \ 2\  \ 2\  \ 2\  \ 2\  \ 2\  \ 2\  \ 2
 \end{bmatrix}.$$
 By Proposition \ref{finitecodes}, there are exactly $8$ non-equivalent codes. Observe now that the monomials $(yz^2)^{b}, (xyz)^{b} \in S_{0,b}$ and $x(yz^2)^{b}, x(xyz)^{b} \in S_{1,b}$. By evaluating these monomials at $Y$, we get the codewords $(1\  \ 1\  \ 1\  \ 1\  \ 1\  \ 1\  \ 1\  \ 1)$ and  $(1\  \ 1\  \ 1\  \ 1\  \ 4\  \ 4\  \ 4\  \ 4)$ which shows that the minimum distance is at most $4$.

According to Markus Grassl's Code Tables \cite{Gra} a best-known code with $N=8$ has $k+d=8$ or $k+d=9$ (MDS codes). So, there is no need to consider codes of degrees $(0,1)$ and $(1,0)$. As the codes of $(0,0)$ and $(1,3)$ are trivially MDS, we look at the codes of degrees $(1,1)$, $(0,2)$, $(1,2)$ and $(0,3)$.
 
 Take $\alpha=(1,1)$ first and consider the toric complete intersection code $\cC_{\aa,Y}$. A basis of $(S/I(Y))_\aa$ is given by the classes of monomials $\{xyz^2, yz^3, xw, zw\}$. So, the generating matrix of the code which is obtained by evaluating these monomials at the $8$ points of $Y$ is as follows:
 $$\small\begin{bmatrix}
    1\  \ 1\  \ 1\  \ 1\  \ 1\  \ 1\  \ 1\  \ 1\\
    1\  \ 1\  \ 1\  \ 1\  \ 4\  \ 4\  \ 4\  \ 4\\
      1\  \ 2\  \ 3\  \ 4\  \ 1\  \ 2\  \ 3\  \ 4\\
      1\  \ 2\  \ 3\  \ 4\  \ 4\  \ 3\  \ 2\  \ 1
 \end{bmatrix}.$$
 It is now easy to check that the minimum distance is $3$ and thus the code $\cC_{\aa,Y}$ has parameters $[8,4,3]_5$, by using Magma, see \cite{magma}. In contrast, the best known code according to \cite{Gra} has parameters $[8,5,3]_5$.

For degree $\alpha=(0,2)$, a basis of $(S/I(Y))_\aa$ is given by the classes of monomials $\{ y^2z^4,xy^2z^3, yz^2w, xyzw,  w^2\}$. %and the generating matrix is as follows:
% $$\begin{bmatrix}
%    1\  \ 1\  \ 1\  \ 1\  \ 1\  \ 1\  \ 1\  \ 1\\
%    1\  \ 1\  \ 1\  \ 1\  \ 4\  \ 4\  \ 4\  \ 4\\
%1\  \ 2\  \ 3\  \ 4\  \ 1\  \ 2\  \ 3\  \ 4\\
%      1\  \ 2\  \ 3\  \ 4\  \ 4\  \ 3\  \ 2\  \ 1\\
%      1\  \ 4\  \ 4\  \ 1\  \ 1\  \ 4\  \ 4\  \ 1
% \end{bmatrix}.$$
 In this case, the code $\cC_{\aa,Y}$ has parameters $[8,5,3]_5$ which is a best possible code according to \cite{Gra}. 
 
For degree $\alpha=(1,2)$, a basis of $(S/I(Y))_\aa$ is given by the classes of monomials $\{ y^2z^5,xy^2z^4, yz^3w, xyz^2w,  zw^2,xw^2\}$. %and the generating matrix is as follows:
% $$\begin{bmatrix}
%    1\  \ 1\  \ 1\  \ 1\  \ 1\  \ 1\  \ 1\  \ 1\\
%    1\  \ 1\  \ 1\  \ 1\  \ 4\  \ 4\  \ 4\  \ 4\\
%1\  \ 2\  \ 3\  \ 4\  \ 1\  \ 2\  \ 3\  \ 4\\
%      1\  \ 2\  \ 3\  \ 4\  \ 4\  \ 3\  \ 2\  \ 1\\
%      1\  \ 4\  \ 4\  \ 1\  \ 1\  \ 4\  \ 4\  \ 1\\
%      1\  \ 4\  \ 4\  \ 1\  \ 4\  \ 1\  \ 1\  \ 4
% \end{bmatrix}.$$
 In this case, the code $\cC_{\aa,Y}$ has parameters $[8,6,2]_5$ which is a best possible code according to \cite{Gra}.
 
Finally, if $\alpha=(0,3)$, a basis of $(S/I(Y))_\aa$ is given by the classes of monomials $\{ y^3z^6,xy^3z^5, y^2z^4w, xy^2z^3w,  yz^2w^2,xyzw^2,w^3\}$. %and the generating matrix is as follows:
% $$\begin{bmatrix}
%    1\  \ 1\  \ 1\  \ 1\  \ 1\  \ 1\  \ 1\  \ 1\\
%    1\  \ 1\  \ 1\  \ 1\  \ 4\  \ 4\  \ 4\  \ 4\\
%1\  \ 2\  \ 3\  \ 4\  \ 1\  \ 2\  \ 3\  \ 4\\
%      1\  \ 2\  \ 3\  \ 4\  \ 4\  \ 3\  \ 2\  \ 1\\
%      1\  \ 4\  \ 4\  \ 1\  \ 1\  \ 4\  \ 4\  \ 1\\
%      1\  \ 4\  \ 4\  \ 1\  \ 4\  \ 1\  \ 1\  \ 4\\
%      1\  \ 3\  \ 2\  \ 4\  \ 1\  \ 3\  \ 2\  \ 4
% \end{bmatrix}.$$
 In this case, the code $\cC_{\aa,Y}$ has parameters $[8,7,2]_5$ which is an MDS code.
 \end{ex}

 As a direct consequence of Proposition \ref{P:CodeHF}, Proposition \ref{P:HilbFla}, and Corollary \ref{HF} 
 we obtain a nice formula for the dimension of the code $\cC_{\aa,Y}$.
 
 \begin{tm}\label{T:CodeFla}  Let $Y \subseteq (\F_q^*)^n$ be a reduced complete intersection subscheme of $X$ such that $I(Y)$ is generated by $n$ homogeneous polynomials with semi-ample degrees $\aa_1,\dots,\aa_n$. Then for any $\aa\in\N\beta$ we have
$$
\dim_{\F_q}(\cC_{\aa,Y})=\sum_{s=0}^{n}\,\, (-1)^s\!\!\! \sum_{I\subseteq [n], |I|=s} |P_{\aa-\aa_I} \cap M |,
$$
where $\aa_I=\sum_{i\in I}\aa_i$.  In particular, if $\aa\not\succeq \aa_i$ for all $1\leq i\leq n$ (i.e. $P_\aa$ does not contain any of $P_{\aa_i}$) then
$$\dim_{\F_q}(\cC_{\aa,Y})=|P_\aa\cap M|.$$
Furthermore, if $\aa \succeq \aa_1+\dots+\aa_n$ then 
the dimension of $\cC_{\aa,Y}$ equals its length $|Y|$ and, hence, the code is trivial. 
\end{tm}

\begin{rema} \cite{sop} gives a bound on the minimum distance of evaluation codes $\cC_{\aa,Y}$ for local complete intersections $Y$ and for degrees less than or equal to the {\it critical degree}, i.e. for
$$\aa \preceq \rho:= \aa_1+\dots+\aa_n-\sum_{j=1}^r\bb_j.$$ So, for degrees between $\rho$ and
$ \aa_1+\dots+\aa_n$, the code $\cC_{\aa,Y}$ may be non-trivial (see examples in the third section) and currently no non-trivial bound for the minimum distance is known.
\end{rema}

We finish with an example illustrating the formula in \rt{CodeFla}.
%\section{Examples}
%$F_1=x^{q-1}-z^{q-1}$ and $F_2=x^{2\ell(q-1)}y^{q-1}-w^{q-1}$

\begin{ex} Let $\F_5$ be a field of 5 elements. Consider a
 3-dimensional toric variety $X$ defined by the normal fan to 
the polytope with vertices $$\{(0,0,0),(0,1,0)(1,-1,1),(1,2,1),(-1,2,1),(-1,-1,1)\}.$$
The corresponding exact sequence $\mathfrak{P}$ is:
$$\dis \xymatrix{ \mathfrak{P}: 0  \ar[r] & \Z^3 \ar[r]^{\phi} & \Z^5 \ar[r]^{\deg}& \Z^2 \ar[r]& 0},$$  
where 
$$\phi=\begin{bmatrix}
 ~1 & -1 & 0 & ~~~0 & ~~~0~ \\
 ~0 & ~~~0 & 1 & -1 & ~~~0~\\
~1 & ~~~1 & 1 & ~~~1 & -1~
\end{bmatrix}^T   \quad  \mbox{ and} \quad 
\deg=\begin{bmatrix}
-1 & -1 & 1 & 1& 0~\\
~~~1& ~~~1 & 0 & 0 & 2~
\end{bmatrix}.
$$
We let $S=\K[x,y,z,w,t]$ be the homogeneous coordinate ring of $X$. 
Now consider the homogeneous ideal $J=\la F_1, F_2, F_3 \ra$, where $F_1=x^4-y^4$,
$F_2=z^4-w^4$, and $F_3=(xyzw)^4-t^4$. The corresponding degrees are 
$\aa_1=(-4,4)$, $\aa_2=(4,0)$, and $\aa_3=(0,8)$. The ideal $J$ is radical and $B$-saturated
and defines a reduced subscheme $Y$ with $|Y|=64$. The polytopes
$P_{\aa_1}$ and $P_{\aa_2}$ are lattice segments of length $4$ with vertices $\{(-2,0,0),(2,0,0)\}$
and $\{(0,-2,0),(0,2,0)\}$, respectively. The polytope  $P_{\aa_3}$ is a square pyramid with
vertices $$\{(0,0,0),(4,-4,4),(4,4,4),(-4,4,4),(-4,-4,4)\}.$$ One can check that the mixed
volume $3!V(P_{\aa_1},P_{\aa_2},P_{\aa_3})=64$.

Now consider $\aa=(-2,7)$. Its polytope $P_\aa$ has vertices 
$$\Big\{(0,0,0),(-7,0,0),(-7,-5,0),(0,-5,0),\Big(-\frac{5}{2},-\frac{5}{2},-\frac{5}{2}\Big),\Big(-\frac{9}{2},-\frac{5}{2},-\frac{5}{2}\Big)\Big\},$$ 
and contains $80$ lattice points.
Furthermore, $P_{\aa-\aa_I}$ is empty for all $I\subset \{1,2,3\}$, but $I=\{1\}$, $I=\{2\}$, and 
$I=\{1,2\}$. 

With the help of {\it polymake} \cite{polymake}, we obtain  $|P_{\aa-\aa_{1}}|=32$, $|P_{\aa-\aa_{2}}|=16$, and  
$|P_{\aa-\aa_{1}-\aa_2}|=8$. Therefore, the formula in \rt{CodeFla} gives
$$\dim_{\F_q}(\cC_{\aa,Y})=80-32-16+8=40.$$
This is confirmed by computing the Hilbert function of $Y$ using {\it Macaulay 2}, see \cite{Mac2}.

\end{ex}

\bibliographystyle{elsart-harv}

\end{document}